\documentclass[final,11pt,a4paper]{article} 

\usepackage[DIV=11,footexclude]{typearea}

\usepackage[utf8]{inputenc}

\usepackage{amsmath, amsthm, amssymb}
\newtheorem{theorem}{Theorem}[section]
\newtheorem{lemma}[theorem]{Lemma}

\newtheorem{cor}[theorem]{Corollary}
\newtheorem{prop}[theorem]{Proposition}

\theoremstyle{definition}
\newtheorem{definition}[theorem]{Definition}
\newtheorem{problem}[theorem]{Problem}
\newtheorem{algo}[theorem]{Algorithm}

\theoremstyle{remark}

\usepackage{tikz}
\tikzstyle{vertex}=[circle,fill=black!100,text=white,inner sep=0.8mm]
\tikzstyle{point}=[circle,fill=black,inner sep=0.1mm]

\usepackage{graphicx} 

\usepackage[british]{babel}

\usepackage{paralist}

\DeclareMathOperator{\Free}{Free}

\def\X{\mathcal{X}}
\def\P{\mathcal{P}}
\def\S{\mathcal{S}}
\def\F{\mathcal{F}}
\def\K{\mathcal{K}}

\long\def\probl#1#2{~\par
\begin{compactitem}
\item[\textsc{Input:}]
#1
\item[\textsc{Output:}]
#2
\end{compactitem}}

\title{Deciding the Bell number for~hereditary~graph~properties%
\thanks{The authors gratefully acknowledge support from DIMAP: the Centre 
for Discrete Mathematics and its Applications at the University of Warwick, 
and from EPSRC, grant EP/L020408/1.} 
\thanks{A conference version of this paper appeared in the
  Proceedings of the 40th International Workshop on Graph-Theoretic Concepts in Computer Science (WG 2014),
  Lecture Notes in Computer Science vol.~8747.}
} 

\author{
Aistis Atminas%
\thanks{Current address: School of Science and Technology,
  Nottingham Trent University, Clifton Campus, Nottingham NG11 8NS, UK.}
\and
Andrew Collins
\and
Jan Foniok%
\thanks{Current address: School of Computing, Mathematics and Digital Technology,
  Manchester Metropolitan University, Chester Street, Manchester M1 5GD, UK.}
\and
Vadim V.~Lozin 
}

\date{{\it\small DIMAP and Mathematics Institute, University of Warwick, Coventry CV4 7AL, UK.}\\~\\ 4th March 2016}
\begin{document}

\maketitle
\begin{abstract}
The paper [J. Balogh, B. Bollob\'{a}s, D. Weinreich, A jump to the Bell number for hereditary graph properties, 
{\it J.\ Combin.\ Theory Ser.\ B\/} 95 (2005) 29--48] identifies a jump in the speed of hereditary graph properties 
to the Bell number~$B_n$ and provides a partial characterisation of the family of minimal classes whose speed 
is at least~$B_n$. In the present paper, we give a complete characterisation of this family. 
Since this family is infinite, the decidability of the problem of determining if the speed of 
a hereditary property is above or below the Bell number is questionable. We answer this question positively 
by showing that there exists an algorithm which, given a finite set~$\F$ of graphs, decides whether 
the speed of the class of graphs containing no induced subgraphs from the set~$\F$ is above or below the Bell number.
For properties defined by infinitely many minimal forbidden induced subgraphs, the speed is known to be above the Bell number.
\end{abstract}


\section{Introduction}
A \emph{graph property} (or a \emph{class of graphs}\footnote{Throughout the paper we use the two terms -- 
graph property and class of graphs -- interchangeably.}) 
is a set of graphs closed under isomorphism.  
Given a property~$\X$, we write $\X_n$ for the number of graphs in~$\X$ 
with vertex set $\{1,2,\ldots,n\}$ (that is, we are counting \emph{labelled} graphs).
Following~\cite{speed}, we call $\X_n$ the \emph{speed} of the property~$\X$. 

A property is \emph{hereditary} if it is closed under taking induced subgraphs.
It is well-known (and can be easily seen) that a graph property~$\X$ is hereditary if and only if 
$\X$~can be described in terms of forbidden induced subgraphs. More formally, for a set~$\F$ of graphs 
we write $\Free(\F)$ for the class of graphs containing no induced subgraph isomorphic to any graph in the set~$\F$. 
A property~$\X$ is hereditary if and only if $\X=\Free(\F)$ for some set~$\F$. We call~$\F$ a set of 
\emph{forbidden induced subgraphs} for~$\X$ and say that graphs in~$\X$ are \emph{$\F$-free}.

The speeds of hereditary properties and their asymptotic structure have been extensively studied, 
originally in the special case of a single forbidden subgraph \cite{EFR86,EKR76,KPR87,PS1,PS3,PS2}, and more recently in general
\cite{Ale92,AloBalBol:The-structure-of-almost,speed,penultimate,jump,SZ94}. These studies showed, in particular, that 
there is a certain relationship between the speed of a property~$\X$ and the structure of graphs in~$\X$, and 
that the rates of the speed growth constitute discrete layers. The first four lower layers have been distinguished
in~\cite{SZ94}: these are constant, polynomial, exponential, and factorial layers. In other words, the authors of~\cite{SZ94}
showed that some classes of functions do not appear as the speed of any hereditary property, and that there are discrete 
jumps, for example, from polynomial to exponential speeds.  

Independently, similar results were obtained 
by Alekseev in~\cite{Ale97}. Moreover, Alekseev provided the first four layers with the description of all minimal classes,
that is, he identified in each layer the family of all classes every proper hereditary subclass of which belongs to a lower layer
(see also~\cite{speed} for some more involved results). In each of the first four lower layers the set 
of minimal classes is finite and each of them is defined by finitely many forbidden induced subgraphs. 
This provides an efficient way of determining whether a property~$\X$ belongs to one of the first three layers.

One more jump in the speed of hereditary properties was identified in~\cite{jump} and it separates~-- within the factorial layer~--
the properties with speeds strictly below the Bell number~$B_n$ from those whose speed is at least~$B_n$. 
With a slight abuse of terminology we will refer to these two families of graph properties as properties 
below and above the Bell number, respectively. The importance
of this jump is due to the fact that all the properties below the Bell number are well-structured. 
In particular, all of them have bounded clique-width~\cite{ALR09} and all of them are well-quasi-ordered by the induced 
subgraph relation~\cite{KL11}. From the results in~\cite{speed,KL11} it follows that every hereditary property below the Bell number can be characterised 
by finitely many forbidden induced subgraphs and hence the membership problem for each of them can be decided in polynomial time. 

Even so, very little is known about the boundary separating the two families,
that is, very little is known about the \emph{minimal} classes above the Bell number. Paper~\cite{jump} distinguishes
two cases in the study of this question: the case where a certain parameter associated with each class of graphs 
is finite and the case where this parameter is infinite. In the present paper, we call this parameter \emph{distinguishing number}.
For the case where the distinguishing number is infinite, \cite{jump}~provides a complete description
of minimal classes, of which there are precisely~13. For the case where the distinguishing number is finite, 
\cite{jump}~mentions only one minimal class above the Bell number (linear forests) and leaves 
the question of characterising other minimal classes open.

In the present paper, we give a complete answer to the above open question: we provide a structural 
characterisation of all minimal classes above the Bell number with a finite distinguishing number.
This family of minimal classes is infinite, which makes the problem of deciding whether 
a hereditary class is above or below the Bell number questionable. Nevertheless, for properties 
defined by \emph{finitely many} forbidden induced subgraphs, our characterisation allows us to 
prove decidability of this problem:  we show that there exists an algorithm which, 
given a finite set~$\F$ of graphs, decides whether the class $\Free(\F)$ is above or below the Bell number.  
 
All preliminary information related to the topic of the paper can be found in Section~\ref{sec:prelim}.
In Section~\ref{sec:structure}, we describe the minimal classes above the Bell number.
Finally, in Section~\ref{sec:decid} we present our decidability result.
Section~\ref{sec:con} concludes the paper with an open problem.

\section{Preliminaries and preparatory results}
\label{sec:prelim}

\subsection{Basic notation and terminology}

All graphs we consider are undirected without multiple edges.
The graphs in our hereditary classes have no loops; however, we allow loops in
some auxiliary graphs, called ``density graphs'' and denoted usually by~$H$, that are used to 
represent the global structure of our hereditary classes.

If $G$ is a graph, $V(G)$~stands for its vertex set, $E(G)$~for its edge set and $|G|$ for the number of vertices (the \emph{order}) of~$G$.
The edge joining two vertices $u$ and $v$ is~$uv$ (we do not use any brackets);
$uv$~is the same edge as~$vu$.

If $W\subseteq V(G)$, then $G[W]$~is the subgraph of~$G$ induced by~$W$.
For $W_1,W_2$ disjoint subsets of~$V(G)$ we define $G[W_1,W_2]$ to be the bipartite subgraph of~$G$
with vertex set~$W_1\cup W_2$ and edge set $\{uv: u\in W_1,\ v\in W_2,\ uv\in E(G)\}$.
The \emph{bipartite complement} of $G[W_1,W_2]$ is the bipartite graph in which two vertices 
$u\in W_1$, $v\in W_2$ are adjacent if and only if they are not adjacent in $G[W_1,W_2]$.

The \emph{neighbourhood} $N(u)$ of a vertex~$u$ in~$G$ is the set of all
vertices adjacent to~$u$, and the \emph{degree} of~$u$ is the number of its neighbours. 
Note that if (and only if) there is a loop at~$u$ then $u\in N(u)$.

As usual, $P_n$, $C_n$ and $K_n$ denote
the path, the cycle and the complete graph with $n$ vertices, respectively. 
Furthermore, $K_{1,n}$ is a star (i.e., a tree with $n+1$ vertices one of which has degree~$n$),
and $G_1+G_2$ is the disjoint union of two graphs. In particular, $mK_n$~is the disjoint union of $m$ copies of~$K_n$.

A \emph{forest} is a graph without cycles, i.e., a graph every connected component of which is a tree. 
A \emph{star forest} is a forest every connected component of which is a star, 
and a \emph{linear forest} is a forest every connected component of which is a path.  

A \emph{quasi-order} is a binary relation which is reflexive and transitive.
A \emph{well-quasi-order} is a quasi-order which contains neither
infinite strictly decreasing sequences nor infinite antichains (sets of pairwise incomparable elements).
That is, in a well-quasi-order any infinite sequence of elements contains an infinite increasing subsequence.

\bigskip

Recall that the Bell number~$B_n$, defined as the number of ways
to partition a set of $n$ labelled elements, satisfies the asymptotic
formula $\ln B_n / n = \ln n-\ln\ln n +\Theta(1)$.

Balogh, Bollob\'{a}s and Weinreich~\cite{jump} showed that if the speed
of a hereditary graph property is at least $n^{(1-o(1))n}$, then it
is actually at least~$B_n$; hence we call any such property a
\emph{property above the Bell number}. Note that this includes
hereditary properties whose speed is exactly equal to the Bell numbers
(such as the class of disjoint unions of cliques).

\subsection{$(\ell,d)$-graphs and sparsification}

Given a graph $G$ and two vertex subsets $U, W \subset V(G)$,  define
$\Delta(U,W)=\max \{ |{N(u) \cap W|},\allowbreak |{N(w) \cap U}| : u \in U, w \in W\}$.
With $\overline{N}(u)=V(G) \backslash (N(u) \cup \{u\})$,
let $\overline{\Delta}(U,W)=\max\{|\overline{N}(u) \cap W|,\allowbreak |\overline{N}(w) \cap U|: w \in W, u \in U \}$.
Note that $\Delta(U,U)$ is simply the maximum degree in~$G[U]$.

\begin{definition}
Let $G$ be a graph. A partition $\pi=\{V_1,V_2,\dotsc,V_{\ell'}\}$ of~$V(G)$ is an \emph{$(\ell,d)$-partition}
if $\ell'\leq \ell$ and for each pair of not necessarily distinct integers $i,j \in \{1,2,\dotsc,\ell'\}$ either
$\Delta(V_i, V_j) \leq d$ or $\overline{\Delta}(V_i, V_j) \leq d$.
We call the sets $V_i$ \emph{bags}.
A graph $G$ is an \emph{$(\ell, d)$-graph} if it admits an $(\ell,d)$-partition.
\end{definition}

If $\Delta (V_i, V_j) \leq d$, we say $V_i$ is \emph{$d$-sparse} with respect to $V_j$,
and if $\overline{\Delta}(V_i, V_j) \leq d$, we say $V_i$ is \emph{$d$-dense} with respect to~$V_j$.
We will also say that the pair $(V_i, V_j)$ is \emph{$d$-sparse} or \emph{$d$-dense}, respectively. 
Note that if the bags are large enough (i.e., $\min\{|V_i|\} > 2d+1$), the terms $d$-dense and $d$-sparse are mutually exclusive. 

\begin{definition}
A \emph{strong $(\ell,d)$-partition} is an $(\ell,d)$-partition each bag of which contains at least $5 \times 2^\ell d$ vertices;
a \emph{strong $(\ell,d)$-graph} is a graph which admits a strong $(\ell,d)$-partition.
\end{definition}

Given any strong $(\ell,d)$-partition $\pi = \{V_1, V_2, \ldots, V_{\ell'}\}$ we define an equivalence relation~$\sim$ 
on the bags by putting $V_i \sim V_j$ if and only if for each~$k$, either $V_k$~is $d$-dense with respect to both $V_i$ and~$V_j$,
or $V_k$~is $d$-sparse with respect to both $V_i$ and~$V_j$.
Let us call a partition $\pi$ \emph{prime} if all its $\sim$-equivalence classes are of size~1.
If the partition $\pi$ is not prime, let $p(\pi)$ be the partition consisting of unions of bags in the $\sim$-equivalence classes
for~$\pi$.

We proceed to showing that the partition $p(\pi)$ of a strong
$(\ell,d)$-graph does not depend on the choice of a strong $(\ell,d)$-partition~$\pi$.
The following three lemmas are the ingredients for the proof of this result.

\begin{lemma} \label{symdif0}
Consider any strong $(\ell,d)$-graph $G$ with any strong $(\ell,d)$-partition $\pi$.
Then $p(\pi)$ is an $(\ell , \ell d)$-partition with at least $5 \times 2^\ell d$ vertices in each bag. 
\end{lemma}

\begin{proof}
Consider two bags $W_1, W_2 \in p(\pi)$. By definition
$W_i=\bigcup_{s \in S_i} V_s$ for some $S_i  \subset \{1,2,\dotsc,\ell'\}$, $i=1,2$.
Also, by the definition of the partition, for all $(s_1, s_2) \in S_1 \times S_2$
the pairs $(V_{s_1}, V_{s_2})$ are all either $d$-dense or $d$-sparse.
If they are $d$-sparse, then for any $s_1 \in S_1$ we have
$\Delta(V_{s_1}, W_2) \leq \sum_{s_2 \in S_2} \Delta(V_{s_1}, V_{s_2}) \leq |S_2|d$.
Since this holds for every $s_1 \in S_1$, for all $x \in W_1$ we have that $|N(x) \cap W_2| \leq |S_2| d$.
Similarly we conclude that for all $x \in W_2$ we have $|N(x) \cap W_1| \leq |S_1| d$.
Therefore, $\Delta(W_1,W_2)\leq \max(|S_1|, |S_2|) d \leq \ell d$.
If the pairs of bags are $d$-dense, a similar argument proves that $\overline{\Delta}(W_1,W_2) \leq \ell d$.
Hence the partition $p(\pi)$ is an $(\ell,\ell d)$-partition.
As it is obtained by unifying some bags from a strong $(\ell,d)$-partition,
we conclude that each bag is of size at least $5 \times 2^\ell d$.
\end{proof}

\begin{lemma}[{\cite[Lemma 10]{jump}}]
\label{symdif1}
Let $G$ be a graph with an $(\ell,d)$-partition $\pi$. If two vertices $x,y \in G$ are in the same bag $V_k$,
then the symmetric difference of their neighbourhoods $N(x) \ominus N(y)$ is of size at most $2\ell d$.
\qed
\end{lemma}

\begin{lemma} \label{symdif2}
Let $G$ be a graph with a strong $(\ell,d)$-partition $\pi$.
If two vertices $x, y \in V(G)$ belong to different bags of the partition $p(\pi)$,
then the symmetric difference of their neighbourhoods $N(x) \ominus N(y)$ is of size at least $5 \times 2^{\ell}d-2d$.  
\end{lemma}

\begin{proof}
Take any two vertices $x \in V_i$ and $y \in V_j$ with bags $V_i$
and $V_j$ belonging to different $\sim$-equivalence classes.
Then there is a bag $V_k$ such that one of the pairs $(V_i,V_k)$ and
$(V_j,V_k)$ is $d$-dense and the other one is $d$-sparse;
without loss of generality, suppose that $(V_i,V_k)$ is $d$-sparse and $(V_j,V_k)$ is $d$-dense.
Then, in particular, $|N(x) \cap V_k| \leq d$ and $|N(y) \cap V_k| \geq |V_k|-d$.
Hence $|N(x) \ominus N(y)| \geq |N(y) \setminus N(x)| \geq |V_k|-2d \geq 5 \times 2^\ell d-2d$.
\end{proof}

We are now ready to prove the uniqueness of $p(\pi)$.  

\begin{theorem}\label{partitions}
Let $G$ be a strong $(\ell,d)$-graph with strong $(\ell,d)$-partitions $\pi$ and $\pi'$. Then $p(\pi)=p(\pi')$.   
\end{theorem}

\begin{proof}
Assume two vertices $x, y \in V(G)$ are in the same bag of the partition $p(\pi)$.
By Lemma~\ref{symdif0}, $p(\pi)$~is an $(\ell, \ell d)$-partition, so
applying Lemma~\ref{symdif1} to~$p(\pi)$ we obtain
$|N(x) \ominus N(y)| \leq 2\ell(\ell d)=2\ell^2d < 5 \times 2^\ell d-2d$.
Thus by Lemma~\ref{symdif2}, $x$ and~$y$ are in the same bag of~$p(\pi')$.
Hence, using symmetry, $x$~and $y$ are in the same bag of~$p(\pi)$
if and only if they are in the same bag of~$p(\pi')$.
We deduce that the partitions are the same, i.e., $p(\pi)=p(\pi')$.
\end{proof}

With any strong $(\ell,d)$-partition $\pi=\{V_1,V_2,\dotsc,V_{\ell'}\}$ of a graph~$G$
we can associate a \emph{density graph} (with loops allowed) $H=H(G,\pi)$:
the vertex set of~$H$ is $\{1,2,\dotsc,\ell'\}$ and there is an edge
joining $i$ and $j$ if and only if $(V_i,V_j)$ is a $d$-dense pair
(so there is a loop at~$i$ if and only if $V_i$~is $d$-dense).

For a graph~$G$, a vertex partition $\pi=\{V_1,V_2,\dotsc,V_{\ell'}\}$ of~$G$
and a graph~$H$ (with loops allowed) with vertex set $\{1,2,\dotsc,\ell'\}$, 
we define (as in~\cite{speed}) the \emph{$H,\pi$-transform} $\psi(G,\pi,H)$ 
to be the graph obtained from~$G$ by replacing $G[V_i, V_j]$ with its 
bipartite complement for every pair $(V_i, V_j)$ for which $ij$~is 
an edge of~$H$, and replacing $G[V_i]$ with its complement for every~$V_i$ 
for which there is a loop at the vertex~$i$ in~$H$.

Moreover, if $\pi$ is a strong $(\ell,d)$-partition we define
$\phi(G,\pi)=\psi(G,\pi,H(G,\pi))$.
Note that $\pi$ is a strong $(\ell,d)$-partition for $\phi(G,\pi)$ and
each pair $(V_i,V_j)$ is $d$-sparse in~$\phi(G,\pi)$. We now show that
the result of this ``sparsification'' does not depend on the initial
strong $(\ell,d)$-partition.

\begin{prop}
\label{prop:phi}
Let $G$ be a strong $(\ell,d)$-graph. Then for any two strong
$(\ell,d)$-partitions $\pi$ and $\pi'$, the graph $\phi(G, \pi)$
is identical to $\phi(G, \pi')$.
\end{prop}

\begin{proof}
Suppose that $\pi=\{U_1,U_2,\dotsc,U_{\hat \ell}\}$ and $\pi'=\{V_1,V_2,\dotsc,V_{\hat \ell'}\}$.
By Theorem~\ref{partitions},
$p(\pi)=p(\pi')=\{W_1,W_2,\dotsc,W_{\hat \ell''}\}$.
Consider two vertices $x$,~$y$ of~$G$.
Let $i,j,i',j',i'',j''$ be the indices such that
$x\in U_i$, $x\in V_{i'}$, $x\in W_{i''}$, $y\in U_j$, $y\in V_{j'}$, $y\in W_{j''}$.
As the partitions have at least $5 \times 2^\ell d$ vertices in each
bag, $\ell d$-dense and $\ell d$-sparse are mutually exclusive properties.
Hence the pair $(U_i,U_j)$ is $d$-sparse if and only if
$(W_{i''},W_{j''})$ is $\ell d$-sparse if and only if
$(V_{i'},V_{j'})$ is $d$-sparse; and analogously for dense pairs.
Therefore $ij\in E(H(G,\pi))$ if and only if $i''j''\in E(H(G,p(\pi)))$ if and only if
$i'j'\in E(H(G,\pi'))$.
We conclude that $xy$~is an edge of $\phi(G,\pi)$ if and only if
it is an edge of~$\phi(G,\pi')$.
\end{proof}

Proposition~\ref{prop:phi} motivates the following definition, originating from~\cite{speed}.

\begin{definition}
\label{def:phi}
For a strong $(\ell,d)$-graph~$G$, its \emph{sparsification} is
$\phi(G)=\phi(G, \pi)$ for any strong $(\ell,d)$-partition~$\pi$ of~$G$.
\end{definition}

\subsection{Distinguishing number $k_{\X}$}

In this section, we discuss the distinguishing number of a hereditary
graph property, which is an important parameter introduced by Balogh,
Bollob\'{a}s and Weinreich in~\cite{speed}.

Given a graph $G$ and a set $X=\{v_1, \dotsc, v_t\} \subseteq V(G)$,
we say that the disjoint subsets $U_1$,~\dots,~$U_m$ of~$V(G)$ are
\emph{distinguished} by~$X$ if for each~$i$, all vertices of~$U_i$
have the same neighbourhood in $X$, and for each $i \neq j$,
vertices $x \in U_i$ and $y \in U_j$ have different neighbourhoods in~$X$.
We also say that $X$~\emph{distinguishes} the sets $U_1$, $U_2$,~\dots,~$U_m$.

\begin{definition}
\label{def:dist}
Given a hereditary property $\X$, we define the \emph{distinguishing number}~$k_{\X}$ as follows:
\begin{compactenum}[(a)]
\item
If for all $k, m \in \mathbb{N}$ we can find a graph $G \in\X$ that admits
some $X \subset V(G)$ distinguishing at least $m$~sets, each  of size at least~$k$,
then put $k_{\X}=\infty$.
\item
Otherwise, there must exist a pair $(k, m)$ such that any vertex subset
of any graph $G \in \X$ distinguishes at most $m$~sets of size at least~$k$.
We define $k_{\X}$ to be the minimum value of~$k$ in all such pairs.
\end{compactenum}
\end{definition}

In \cite{speed}, Balogh, Bollob\'{a}s and Weinreich show that the speed
of any hereditary property~$\X$ with $k_\X=\infty$ is above the Bell number.
To study the classes with $k_\X<\infty$, in the next sections we will need two results from their paper.

\begin{lemma}[\cite{speed}, Lemma 27]\label{lem:2.10}
If $\X$ is a hereditary property with finite distinguishing number~$k_{\X}$,
then there exist absolute constants $\ell_{\X}$, $d_{\X}\le k_{\X}$ and $c_{\X}$
such that for all $G \in {\X}$, the graph~$G$ contains an induced subgraph~$G'$
such that $G'$~is an $(\ell_{\X}, d_{\X})$-graph and ${|V(G) \backslash V(G')| < c_{\X}}$.
\qed
\end{lemma}

By removing all the small bags with fewer than $5\times 2^{\ell_\X} d_\X$ vertices,
which affects only the constant~$c_\X$,
we can actually assume that the graph~$G'$ is a \emph{strong} $(\ell_\X, d_\X)$-graph.
This observation allows us to strengthen Lemma~\ref{lem:2.10} as follows.

\begin{lemma} \label{finitedistinguishingnumber}
If $\X$ is a hereditary property with finite distinguishing number~$k_{\X}$,
then there exist absolute constants $\ell_{\X}$, $d_{\X}$ and $c_{\X}$
such that for all $G \in {\X}$, the graph~$G$ contains an induced subgraph~$G'$
such that $G'$~is a strong $(\ell_{\X}, d_{\X})$-graph and ${|V(G) \backslash V(G')| < c_{\X}}$.
\qed
\end{lemma}

Finally, we will use this theorem:

\begin{theorem}[\cite{speed}, Theorem 28] \label{paths}
Let $\X$ be a hereditary property with $k_\X<\infty$.
Then $\X_n \geq n^{(1+o(1))n}$ if and only if for every $m$ there
exists a strong $(\ell_\X, d_\X)$-graph~$G$ in~$\X$ such that its sparsification~$\phi(G)$
has a connected component of order at least~$m$.
\qed
\end{theorem}


\section{Structure of minimal classes above Bell}
\label{sec:structure}

In this section, we describe minimal classes with speed above the Bell number.
In~\cite{jump}, Balogh, Bollob\'{a}s and Weinreich characterised
all minimal classes with infinite distinguishing number.  In
Section~\ref{sec:infinite}
we report this result and prove additionally that each of these classes can be
characterised by finitely many forbidden induced subgraphs. Then in Section~\ref{sec:finite}
we move on to the case of finite distinguishing number, which had been
left open in~\cite{jump}.

\subsection{Infinite distinguishing number}
\label{sec:infinite}

\begin{theorem}[Balogh--Bollob\'{a}s--Weinreich~\cite{jump}]
\label{thm:infinite}
Let $\X$ be a hereditary graph property with $k_\X=\infty$.
Then $\X$ contains at least one of the following (minimal) classes:
\begin{compactenum}[(a)]
\item the class~$\K_1$ of all graphs each of whose connected components is a clique;
\item the class~$\K_2$ of all star forests;
\item the class~$\K_3$ of all graphs whose vertex set can be split into an independent set~$I$
	and a clique~$Q$ so that every vertex in~$Q$ has at most one neighbour in~$I$;
\item the class~$\K_4$ of all graphs whose vertex set can be split into an independent set~$I$
	and a clique~$Q$ so that every vertex in~$I$ has at most one neighbour in~$Q$;
\item the class~$\K_5$ of all graphs whose vertex set can be split into two cliques~$Q_1$,~$Q_2$
	so that every vertex in~$Q_2$ has at most one neighbour in~$Q_1$;
\item the class $\K_6$ of all graphs whose vertex set can be split into two independent sets~$I_1$,~$I_2$
	so that the neighbourhoods of the vertices in~$I_1$ are linearly ordered by inclusion
	(that is, the class of all \emph{chain graphs});
\item the class $\K_7$ of all graphs whose vertex set can be split into an independent set~$I$
	and a clique~$Q$ so that the neighbourhoods of the vertices in~$I$ are linearly ordered by inclusion
	(that is, the class of all \emph{threshold graphs});
\item the class~$\overline{\K_i}$ of all graphs whose complement belongs to~$\K_i$ as above,
	for some~$i\in\{1,2,\dotsc,6\}$ (note that the complementary class of~$\K_7$ is $\K_7$ itself).
\end{compactenum}
\end{theorem}

As an aside, it is perhaps worth noting that each of the minimal classes
admits an infinite universal graph.
To be specific, $\K_1$~is the age (the class of all finite induced subgraphs)
of~$U_1$, the disjoint union of $\omega$ cliques, each of order~$\omega$.
The remaining universal graphs are depicted in Figure~\ref{fig:uni};
a grey oval indicates a clique (of order~$\omega$).

\begin{figure}[h]
\centerline{\includegraphics{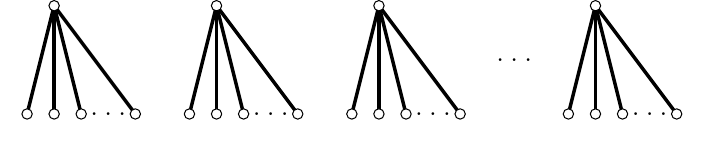}\hfil
\includegraphics{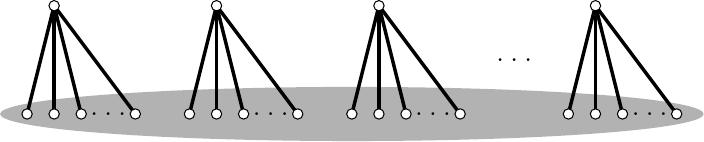}}
\centerline{$U_2$\hfil\hfil$U_3$}
\bigskip
\centerline{\includegraphics{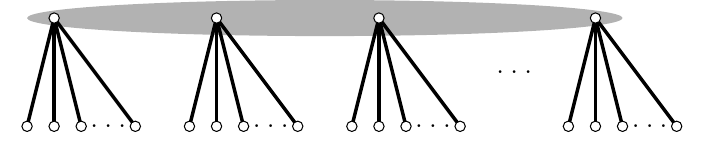}\hfil
\includegraphics{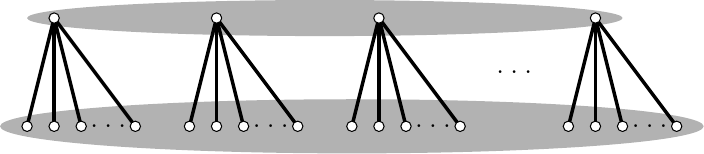}}
\centerline{$U_4$\hfil\hfil$U_5$}
\bigskip
\centerline{\includegraphics{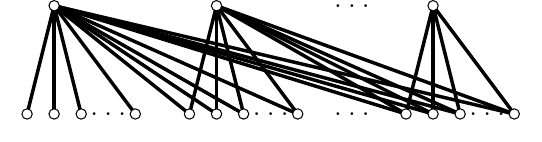}\hfil
\includegraphics{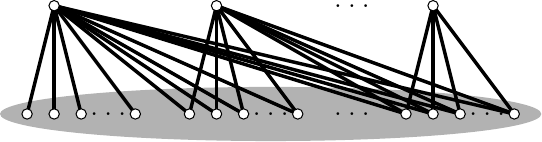}}
\centerline{$U_6$\hfil\hfil$U_7$}
\caption{The universal graphs.}
\label{fig:uni}
\end{figure}

Aiming to prove that each of the classes above is defined by
forbidding finitely many induced subgraphs, we first state an older
result by Földes and Hammer about split graphs of which we make use in our proof.
A \emph{split graph} is a graph whose vertex set can be split into
an independent set and a clique.

\begin{theorem}[\cite{FolHam:Split-graphs}]
\label{thm:split}
The class of all split graphs is exactly the class $\Free(2K_2,C_4,C_5)$.
\end{theorem}

Before showing the characterisation of the classes $\K_1$--$\K_6$ in
terms of forbidden induced subgraphs, we introduce some of the less commonly
appearing graphs:
the \emph{claw}~$K_{1,3}$,
the \emph{3-fan}~$F_3$,
the \emph{diamond}~$K_4^-$,
and the graph~$H_6$ (Fig.~\ref{fig:small}).

\begin{figure}[h]
	\begin{center}
		\begin{tikzpicture}
	  		[scale=1,auto=left]
			\node[vertex] (x1) at (0,1)   { };
			\node[vertex] (x2) at (1,2)   { };
			\node[vertex] (x3) at (1,1) { };
			\node[vertex] (x4) at (1,0) { };
			
			\foreach \from/\to in {x1/x2,x1/x3,x1/x4}
	    	\draw (\from) -- (\to);
			\coordinate [label=center:$K_{1,3}$] (K_{1,3}) at (0.5,-1);			
	
			\node[vertex] (y1) at (3,0)   { };
			\node[vertex] (y2) at (3,2)   { };
			\node[vertex] (y3) at (4,1) { };
			\node[vertex] (y4) at (5,0) { };
			\node[vertex] (y5) at (5,2)   { };
					
			\foreach \from/\to in {y1/y2,y1/y3,y1/y4,y2/y3,y3/y4,y3/y5,y4/y5}
	    	\draw (\from) -- (\to);
			\coordinate [label=center:$F_3$] (F_3) at (4,-1);

			\node[vertex] (z1) at (7,1)   { };
			\node[vertex] (z2) at (8,2)   { };
			\node[vertex] (z3) at (8,0) { };
			\node[vertex] (z4) at (9,1) { };
					
			\foreach \from/\to in {z1/z2,z1/z3,z1/z4,z2/z4,z3/z4}
	    	\draw (\from) -- (\to);
			\coordinate [label=center:$K_4^-$] (K_4^-) at (8,-1);

			\node[vertex] (w1) at (11,2)   { };
			\node[vertex] (w2) at (11,0)   { };
			\node[vertex] (w3) at (12,1) { };
			\node[vertex] (w4) at (13,1) { };
			\node[vertex] (w5) at (14,2)   { };
			\node[vertex] (w6) at (14,0)   { };
			
			\foreach \from/\to in {w1/w3,w2/w3,w3/w4,w4/w5,w4/w6}
	    	\draw (\from) -- (\to);
			\coordinate [label=center:$H_6$] (H_6) at (12.5,-1);			

		\end{tikzpicture}
	\end{center}
	\caption{Some small graphs}
\label{fig:small}
\end{figure}
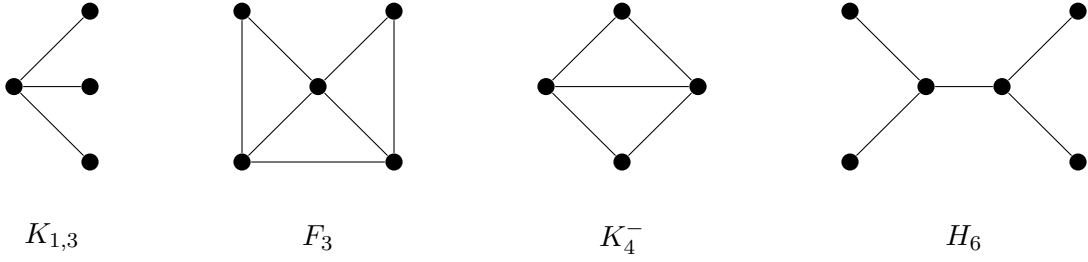

\begin{theorem}
\label{thm:finforb}
Each of the classes of Theorem~\ref{thm:infinite} is defined by finitely many forbidden induced subgraphs.
\end{theorem}

\begin{proof}
First, observe that if we define $\overline\X$ as the class of the
complements of all graphs in~$\X$, then $\overline{\Free(\F)}=\Free(\overline\F)$.
Hence if each class~$\K_i$ is defined by finitely many forbidden induced subgraphs, then so is each~$\overline{\K_i}$.

(a) $\K_1=\Free(P_3)$: It is trivial to check that $P_3$ does not belong to $\K_1$, 
and any graph not containing an induced $P_3$ must be a collection of cliques.

(b) $\K_2=\Free(K_3,P_4,C_4)$: Obviously, none of the graphs $K_3$, $P_4$, $C_4$ belongs to~$\K_2$.
Let $G\in\Free(K_3,P_4,C_4)$. Since every cycle of length at least~$5$
contains~$P_4$, $G$~does not contain any cycles; thus $G$~is a forest.
The absence of a $P_4$ implies that the diameter of any connected component of~$G$ is at most~$2$,
hence $G$~is a star forest.

(c) $\K_3=\Free(\F)$ for $\F=\{2K_2, C_4, C_5, K_{1,3}, F_3\}$:
It is easy to check that none of the forbidden graphs belong to~$\K_3$. 
Let $G\in\Free(\F)$. By Theorem~\ref{thm:split}, $G$~is a split graph.
Split~$G$ into a maximal clique~$Q$ and an independent set~$I$. Suppose, for the sake of contradiction,
that $Q$~contains a vertex~$u$ with two neighbours~$a,b\in I$.
As we took $Q$ to be a maximal clique, $a$~has a non-neighbour~$v$
and $b$~has a non-neighbour~$w$ in~$Q$.  If $a$, $w$ are not adjacent,
then the vertices $a,b,u,w$ induce a claw in~$G$; if $b$, $v$ are
not adjacent, then the vertices $a,b,u,v$ induce a claw in~$G$;
otherwise the vertices $a,b,u,v,w$ induce a 3-fan in~$G$. In either case
we get a contradiction.

(d) $\K_4=\Free(\F)$ for $\F=\{2K_2, C_4, C_5, K_4^-\}$:
Again, it is easy to check that none of the forbidden graphs belong to~$\K_4$. 
Let $G\in\Free(\F)$. By Theorem~\ref{thm:split}, $G$~is a split graph.
Just like before, split $G$ into a maximal clique~$Q$ and an independent set~$I$.
Suppose that some vertex~$u$ in~$I$ has two neighbours~$a,b$ in~$Q$. By maximality
of~$Q$, $u$~also has a non-neighbour~$c$ in~$Q$. But then the vertices $a,b,c,u$
induce a~$K_4^-$ in~$G$, a contradiction.

(e) The class~$\overline{\K_5}$ of the complements of the graphs
in~$\K_5$ is characterised as the class of all (bipartite) graphs
whose vertex set can be split into independent sets~$I_1$,~$I_2$
so that each vertex in~$I_2$ has at most one non-neighbour in~$I_1$.
We show that $\overline{\K_5}=\Free(\F)$ for $\F=\{K_3, C_5, P_4+K_1, 2K_2+K_1, C_4+K_2, C_4+2K_1, H_6\}$.
The reader will kindly check that indeed no graph in~$\F$ belongs to~$\overline{\K_5}$.

Consider some $G\in\Free(\F)$; we will show that $G\in\overline{\K_5}$. Observe that $\F$~prevents~$G$
from having an odd cycle, thus $G$~is bipartite. We distinguish
three cases depending on the structure of the connected components of~$G$.

First, suppose that $G$ has at least two non-trivial connected components
(that is, connected components that are not just isolated vertices).
Because $G$~is $(2K_2+K_1)$-free, it only has two connected components in all.
Being $C_4$- and $P_4$-free, each component is necessarily a star. 
Observe that any graph consisting of one or two stars belongs to~$\overline{\K_5}$.

Next assume that $G$ has only one non-trivial connected component and some isolated vertices.
The non-trivial component is bipartite and $P_4$-free, so it is a biclique.
If this biclique contains~$C_4$, then $G$~only contains one other isolated vertex;
any graph consisting of a biclique and one isolated vertex is in~$\overline{\K_5}$.
Otherwise the biclique is a star; any graph consisting of a star and one or more
isolated vertices belongs to~$\overline{\K_5}$.

Finally, consider $G$ that is connected. We will show that for any two vertices of $G$ in different parts, 
one of them must have at most one non-neighbour in the opposite part. Suppose this is not true and there are 
$x,y\in V(G)$ in different parts such that both $x$ and~$y$ have at least two non-neighbours in the opposite part. 
Assume first that $x$ and~$y$ are adjacent.
Let $a$ and $b$ be two non-neighbours of~$x$, and let $c$ and $d$ be two non-neighbours of~$y$. 
Then the graph induced by $a,b,c$ and $d$ cannot be a $C_4$, $P_4$, $P_3+K_1$, $2K_2$ or $K_2+2K_1$, 
because $G$ is $(P_4+K_1, 2K_2+K_1, C_4+K_2)$-free. Hence $a,b,c$ and $d$ must induce a $4K_1$. 
As $G$ is connected, $a$ must have a neighbour, say $w$. However, the vertices $x,y,a,c$ and $w$ induce 
a $P_4+K_1$ if $y$ and $w$ are adjacent and they induce a $2K_2+K_1$ if $y$ and $w$ are not adjacent. 
Therefore, $x$ and $y$ must be non-neighbours.

\looseness-1
By assumption, $x$~has another non-neighbour~$a\ne y$ in the opposite part,
and $y$~has another non-neighbour~$b\ne x$ in the opposite part.
As $G$ is connected, $x$ must have a neighbour, say $u$. If $a$ and $b$ are adjacent, then  $x,y,u,a$ and $b$ induce a $2K_2+K_1$ 
if $u$ is not adjacent to $b$, and they induce a $P_4+K_1$ if $u$ is adjacent to $b$. 
Both cases lead to a contradiction as $G$ is $(P_4+K_1, 2K_2+K_1)$-free, hence $a$ and $b$ cannot be adjacent. 
Now, as $G$ is connected, $y$ must also have a neighbour, say $v$. 
If $u$ is not adjacent to $b$, then $x,y,u,v$ and $b$ induce either a $2K_2+K_1$ or a $P_4+K_1$, hence $u$ and $b$ must be adjacent. 
By a symmetric argument, $v$ is adjacent to $a$.  
Now $u$ and $v$ must be non-adjacent: otherwise $x,y,u,v,a$ and $b$ induce an $H_6$.

This argument shows that any neighbour of $x$ must also be a neighbour of $b$,  any neighbour of $y$ must also be a neighbour of $a$, 
and that any neighbour of $x$ cannot be adjacent to any neighbour of $y$. This means that the shortest induced path between $x$ and $y$ 
must contain a $P_6$, which is a contradiction as $G$ is $(P_4+K_1)$-free. Therefore, either $x$ or $y$ must have at most one non-neighbour. 
This implies that $G$~can be split into two independent sets $I_1$, $I_2$
such that every vertex in $I_2$ has at most one non-neighbour in~$I_1$,
so $G$~belongs to~$\overline{\K_5}$.

(f) Chain graphs are characterised by finitely many forbidden induced subgraphs by a result of
Yannakakis~\cite{Yan:The-complexity-of-the-partial}; namely, $\K_6=\Free(2K_2, K_3, C_5)$.

(g) Threshold graphs are characterised by finitely many forbidden induced subgraphs by a result of 
Chv\'{a}tal and Hammer~\cite{threshold-graphs};
namely, $\K_7=\Free(2K_2, P_4, C_4)$.
\end{proof}

\subsection{Finite distinguishing number}
\label{sec:finite}

In this section we provide a characterisation of the minimal
classes for the case of finite distinguishing number~$k_\X$.
It turns out that these minimal classes
consist of $(\ell_\X,d_\X)$-graphs, that is, the vertex set of
each graph is partitioned into at most~$\ell_\X$~bags and dense
pairs are defined by a density graph~$H$ (see Lemma~\ref{finitedistinguishingnumber}).
The condition of Theorem~\ref{paths} is enforced by long paths
(indeed, an infinite path in the infinite universal graph).
Thus actually $d_\X\le2$ for the minimal classes~$\X$.

\bigskip

Let $A$ be a finite alphabet.
A \emph{word} is a mapping $w: S \rightarrow A$, where $S=\{1,2,\ldots, n\}$ for some $n \in \mathbb{N}$ or $S=\mathbb{N}$;
$|S|$~is the \emph{length} of~$w$, denoted by~$|w|$.
We write $w_i$ for~$w(i)$, and we often use the notation
$w=w_1w_2w_3\ldots w_n$ or $w=w_1w_2w_3\ldots$.
For $n\le m$ and $w=w_1w_2\ldots w_n$, $w'=w'_1w'_2\ldots w'_m$ (or $w'=w'_1w'_2\ldots$),
we say that $w$~is a \emph{factor} of~$w'$ if there exists
a non-negative integer~$s$ such that $w_i=w'_{i+s}$ for $1\le i\le n$;
$w$~is an \emph{initial segment} of~$w'$ if we can take $s=0$.

Let $H$ be an undirected graph with loops allowed and with vertex set $V(H)=A$,
and let $w$ be a (finite or infinite) word over the alphabet~$A$.
For any increasing sequence $u_1 < u_2 < \dotsb <u_m$ of positive
integers such that $u_m\le|w|$, define $G_{w,H}(u_1, u_2, \ldots, u_m)$ to be the graph
with vertex set $\{u_1, u_2, \dotsc, u_m\}$ and an edge between
$u_i$ and $u_j$ if and only if 
\begin{compactitem}
\item either $|u_i-u_j|=1$ and $w_{u_i}w_{u_j} \notin{E(H)}$,
\item or $|u_i-u_j| > 1$ and $w_{u_i}w_{u_j} \in E(H)$.
\end{compactitem}

Let $G=G_{w,H}(u_1, u_2, \ldots, u_m)$ and define $V_a=\{{u_i} \in V(G): w_{u_i}=a\}$ for any $a \in A$.
Then $\pi=\pi_{w}(G)=\{V_a : a \in A \}$ is an $(|A|,2)$-partition,
and so $G$~is an $(|A|,2)$-graph.
Moreover, $\psi(G,\pi,H)$~is a linear forest whose paths are formed by the 
segments of consecutive integers within the set $\{u_1, u_2, \ldots, u_m\}$.
This partition $\pi_{w}(G)$ is called the \emph{letter partition} of~$G$.

\begin{definition}
Let $H$ be an undirected graph with loops allowed and with vertex set $V(H)=A$,
and let $w$ be an infinite word over the alphabet~$A$.
Define $\P(w, H)$ to be the hereditary class consisting of the
graphs $G_{w, H}(u_1, u_2, \ldots, u_m)$ for all finite increasing
sequences $u_1 < u_2 < \dotsb < u_m$ of positive integers.
\end{definition}

As we shall see later, all classes $\P(w,H)$ are above the Bell number.
More importantly, all minimal classes above the Bell number have the form $\P(w,H)$
for some $w$ and~$H$.
Our goal here is firstly to describe sufficient conditions on the
word~$w$ under which $\P(w,H)$ is a minimal class above the Bell number;
moreover, we aim to prove that any hereditary class above the Bell
number with finite distinguishing number contains the class $\P(w,H)$
for some word~$w$ and graph~$H$.
We start by showing that these classes indeed have finite distinguishing number.

\begin{lemma}\label{lem:aux2}
For any word $w$ and graph $H$ with loops allowed, 
the class $\X=\mathcal{P}(w,H)$ has finite distinguishing number.
\end{lemma}
\begin{proof}
Put $\ell=|H|$ and let $G$ be a graph in $\X$. Consider the letter partition $\pi=\pi_w(G)=\{V_a:a\in V(H)\}$ of~$G$,
which is an $(\ell,2)$-partition. 
Choose an arbitrary set of vertices $X \subseteq V(G)$ and let $\{U_1, U_2, \ldots, U_k\}$ be the sets distinguished by~$X$. 
If there are subsets $U_i$, $U_j$ and $V_a$ such that $|V_a \cap U_i| \geq 3$ and $|V_a \cap U_j| \geq 3$, 
then some vertex of~$X$ has at least three neighbours and at least three non-neighbours in~$V_a$, 
which contradicts the fact that $\pi$ is an $(\ell, 2)$-partition. Therefore, in the partition 
$\{V_a \cap U_i : a\in V(H),\ 1 \leq i \leq k\}$ we have at most $\ell$ sets of size at least~3. 
Note that every set $U_i$ of size at least $2 \ell + 1$ must contain at least one such set. 
Hence the family $\{U_1, U_2, \ldots, U_k\}$ contains at most $\ell$ sets of size at least $2 \ell +1$.
Since the set $X$ was chosen arbitrarily, we conclude that $k_{\X} \le 2\ell+1$, as required.   
\end{proof}

The graphs $G_{w,H}(u_1,u_2,\dotsc,u_n)$ defined on a sequence of
consecutive integers will play a special role in our considerations.

\begin{definition}
If $u_1, u_2, \ldots, u_m$ is a sequence of consecutive integers (i.e., $u_{k+1}=u_k+1$ for each $k$),
we call the graph $G_{w,H}(u_1,u_2, \ldots, u_m)$ an \emph{$|H|$-factor}.
Notice that each $|H|$-factor is an $(|H|,2)$-graph; if its letter partition is a strong $(|H|,2)$-partition,
we call it a \emph{strong $|H|$-factor}.
\end{definition}

Note that if $G=G_{w,H}(u_1,u_2, \ldots, u_m)$ is a strong $|H|$-factor,
then its sparsification $\phi(G)=\psi(G,\pi_{w}(G),H)$ is an induced path with $m$ vertices.

\begin{prop}
\label{prop:pwh}
If $w$ is an infinite word over a finite alphabet~$A$ and $H$~is a graph on~$A$, with loops allowed,
then the class $\P(w,H)$ is above the Bell number.
\end{prop}

\begin{proof}
We may assume that every letter of~$A$ appears in~$w$ infinitely
many times: otherwise we can remove a sufficiently long starting
segment of~$w$ to obtain a word~$w'$ satisfying this condition,
replace~$H$ with its induced subgraph~$H'$ on the alphabet~$A'$
of~$w'$, and obtain a subclass $\P(w',H')$ of~$\P(w,H)$ with that property.
For sufficiently large $k$, the $|A|$-factor $G_k=G_{w,H}(1,\dotsc,k)$
is a strong $|A|$-factor; thus $\phi(G_k)$~is an induced path of length $k-1$.
Having a finite distinguishing number by Lemma~\ref{lem:aux2},
the class $\P(w,H)$ is above the Bell number by Theorem~\ref{paths}.
\end{proof}

\begin{definition}
An infinite word~$w$ is called \emph{almost periodic} if for any factor~$f$
of~$w$ there is a constant~$k_f$ such that any factor of~$w$
of length at least~$k_f$ contains~$f$ as a factor.
\end{definition}

The notion of an almost periodic word plays a crucial role in our 
characterisation of minimal classes above the Bell number. 
First, let us show that if $w$ is almost periodic, then $\mathcal{P}(w,H)$
is a minimal property above the Bell number. To prove this, we need an auxiliary lemma.

\begin{lemma}
\label{lem:conse}
Consider $G=G_{w,H}(u_1,\dotsc,u_n)$. If $G$ is a strong $(\ell,d)$-graph
and $\phi(G)$~contains a connected component~$C$ such that 
$|C| \ge \left(2d'^2 \ell^2 |H|^{2} + 1\right) (m-1) + 1$, where $d'=\max\{d,2\}$,
then $V(C)$~contains a sequence of $m$ consecutive integers.
\end{lemma}

\begin{proof}
Let $\pi = \{U_1, U_2, \ldots, U_{\ell'} \}$ be a strong $(\ell,d)$-partition of~$G$,
so that $\ell'\le\ell$ and $\phi(G)=\phi(G,\pi)$;
let $\pi'=\{V_a: a\in V(H)\}$ be the letter partition of~$G$, given by $V_a=\{u_j\in V(G):w_{u_j}=a\}$.
Put $k=|H|$.
Note that $\pi'$ is an $(k,2)$-partition, hence also an $(k, d')$-partition.

Let $E=E(\phi(G))\setminus E(\psi(G,\pi',H))$ be the set of all the
edges of $\phi(G)$ that are not edges of $\psi(G,\pi',H)$,
that is, that do not join two consecutive integers.
We will now upper-bound the number of such edges.
Observe that $E$~consists of
\begin{inparaenum}[(a)]
\item the edges between $U_i\cap V_{a}$ and $U_j\cap V_{b}$ 
where $(U_i,U_j)$ is $d'$-sparse and $(V_{a},V_{b})$ is $d'$-dense, and
\item the non-edges between $U_i\cap V_{a}$ and $U_j\cap V_{b}$ 
where $(U_i,U_j)$ is $d'$-dense and $(V_{a},V_{b})$ is $d'$-sparse.
\end{inparaenum}

Consider the partition $\rho=\{U_i \cap V_a\colon 1\le i\le\ell',\ a\in V(H)\}$ of~$G$,
which is an $(\ell'k,d')$-partition.
Let $(U_i\cap V_{a}, U_j\cap V_{b})$ be a pair of non-empty sets
such that $(U_i,U_j)$ is $d'$-sparse but $(V_{a},V_{b})$ is $d'$-dense.
Each such pair is both $d'$-sparse and $d'$-dense, and consequently we have
$|U_i\cap V_{a}|\le 2d'$ and $|U_j\cap V_{b}|\le 2d'$. Moreover, there
are at most $2d'^2$ edges between $U_i\cap V_{a}$ and $U_j\cap V_{b}$.
Similarly, for any pair $(U_i\cap V_{a}, U_j\cap V_{b})$ where
$(U_i,U_j)$ is $d'$-dense but $(V_{a},V_{b})$ is $d'$-sparse, we can show
that there are at most $2d'^2$ non-edges between $U_i\cap V_{a}$ and
$U_j\cap V_{b}$.
We conclude that $|E| \le 2d'^2(\ell'k)^2$.

Any edge of~$\phi(G)$ that is not in~$E$ joins two consecutive integers.
Hence any connected component~$C$ of~$\phi(G)$ consists of at most
$|E|+1$ segments of consecutive integers connected by edges from~$E$.
If $C$~does not contain a sequence of $m$ consecutive integers, it consists of
at most $|E|+1 \leq 2d'^2(\ell'k)^2+1$ segments of consecutive integers,
each of length at most~$m-1$; it can therefore contain at most
$\left(2d'^2(\ell'k)^2+1\right)(m-1) \le \left(2d'^2 \ell^2 |H|^2 + 1\right)(m-1)$ vertices.
\end{proof}

\begin{theorem}
\label{thm:ap-min}
If $w$ is an almost periodic infinite word and $H$~is a finite graph with loops allowed, 
then $\mathcal{P}(w,H)$ is a minimal hereditary property above the Bell number.
\end{theorem}

\begin{proof}
The class $\P=\P(w,H)$ is above the Bell number by Proposition~\ref{prop:pwh}.
Thus we only need to show that  any proper hereditary subclass~$\X$ of~$\P$ is below the
Bell number. Suppose $\X \subset \P$ and let $F\in\P\setminus\X$. By
definition of $\P(w,H)$, the graph~$F$ is of the form
$G_{w,H}(u_1,\dotsc,u_n)$ for some positive integers $u_1<\dotsb<u_n$. Let
$w'$ be the finite word $w'=w_{u_1}w_{u_1+1}w_{u_1+2}\dotso
w_{u_n-1}w_{u_n}$. As $w$ is almost periodic, there is an integer~$m$ such
that any factor of~$w$ of length~$m$ contains~$w'$ as factor. Assume, for the sake
of contradiction, that $\X$~is hereditary and above the Bell number. By Lemma~\ref{lem:aux2}, 
the distinguishing number of~$\P$, and hence of~$\X$, is finite, and therefore,
by Lemma~\ref{finitedistinguishingnumber} and Theorem~\ref{paths}, there exists
a strong $(\ell_\X,d_\X)$-graph $G=G_{w,H}(u_1',u_2',\dotsc,u'_{n'})\in\X$ 
such that $\phi(G)$~has a connected component~$C$ of order at least $\left(2d^2 \ell^2
|H|^{2} + 1\right) (m-1) + 1$, where $\ell=\ell_{\X}$ and $d=\max\{d_{\X},2\}$.
By Lemma~\ref{lem:conse}, the vertices of~$C$ contain a sequence of $m$ consecutive integers,
i.e., $V(G)\supseteq V(C)\supseteq \{u',u'+1,\dotsc,u'+m-1\}$.
However, the word $w_{u'}w_{u'+1}\dotso w_{u'+m-1}$ contains~$w'$;
therefore $G$~contains~$F$, a contradiction.
\end{proof}
 
The existence of minimal classes does not necessarily imply that every class above 
the Bell number contains a minimal one. However, in our case this turns out to be true, 
as we proceed to show next. Moreover, this will also imply that the minimal
classes described in Theorem~\ref{thm:ap-min} are the {\it only} minimal classes 
above the Bell number with $k_{\X}$ finite. 
To prove this, we first show in the next two lemmas that any
class~$\X$ above the Bell number with $k_{\X}$ finite contains
arbitrarily large strong $\ell_{\X}$-factors.

\begin{lemma} \label{zero}
Let $\X$ be a hereditary class with speed above the Bell number and
with finite distinguishing number~$k_{\X}$. Then for each~$m$, the
class~$\X$ contains an $\ell_{\X}$-factor of order~$m$.
\end{lemma}

\begin{proof}
From Theorem~\ref{paths} it follows that for each~$m$ there is a
graph $G_m \in {\X}$ which admits a strong $(\ell_{\X}, d_{\X})$-partition
$\{V_1, V_2, \ldots, V_{\ell_m}\}$ with $\ell_m \leq \ell_{\X}$ such that
the sparsification~$\phi(G_m)$ has a connected component~$C_m$ of order at least~$(\ell_{\X}d_{\X})^m$.
Fix an arbitrary vertex~$v$ of~$C_m$.
As $C_m$ is an induced subgraph of~$\phi(G_m)$, the maximum degree in~$C_m$ is bounded by $d=\ell_{\X}d_{\X}$.
Hence for any $k>0$, in~$C_m$ there are at most $d(d-1)^{k-1}$ vertices at distance~$k$ from~$v$;
so there are at most $1+\sum_{k=1}^{m-2}d(d-1)^{k-1} < d^m$
vertices at distance at most $m-2$ from~$v$.
As $C_m$~has order at least~$d^m$, there exists a vertex~$v'$ of distance~$m-1$ from~$v$.
Therefore $C_m$ contains an induced path $v=v_1, v_2, \ldots, v_m=v'$ of length~$m-1$. 

Let $A=\{1,2, \ldots, \ell_m\}$ and let $H$ be the graph with vertex set~$A$
and edge between $i$ and $j$ if and only if $V_i$~is $d_{\X}$-dense with respect to $V_j$.
Let $w_i \in A$ be such that $v_i \in V_{w_i}$ and define the word $w=w_1w_2 \ldots w_m$.
The induced subgraph $G_m[v_1, v_2, \ldots, v_m] \cong G_{w,H}(1, 2, \ldots, m)$
is an $\ell_{\X}$-factor of order~$m$ contained in~$\X$.
\end{proof}

\begin{lemma}\label{first}
Let $\ell$ and $B$ be positive integers such that $B \geq 5 \times 2^{\ell+1}$.
Then any $\ell$-factor $G_{w,H}(1, 2, \ldots, |w|)$ of order at least~$B^{\ell}$ contains
a strong $\ell$-factor $G_{w',H}(1,2, \ldots, |w'|)$
of order at least~$B$ such that $w'$~is a factor of~$w$.
\end{lemma}

\begin{proof}
We will prove by induction on $r \in \{1,2,\dotsc,\ell\}$ that any
$\ell$-factor $G_{w, H}(1,2, \ldots, B^r)$ on $B^r$ vertices with at most $r$ bags in the letter partition contains
a strong $\ell$-factor on at least $B$ vertices. 
For $r=1$ the statement holds because any $\ell$-factor with one bag in the letter partition of order $B \geq 5 \times 2^{\ell+1}$
is a strong $\ell$-factor.
Suppose $1< r \leq \ell$. Then either each letter of~$w=w_1w_2 \ldots w_{B^r}$ appears at
least $B$ times, in which case we are done, or there is a letter
$a=w_i$ which appears less than $B$ times in~$w$.
Consider the maximal factors of~$w$ that do not contain the letter~$a$.
Because the number of occurrences of the letter~$a$ in~$w$ is less than~$B$,
there are at most $B$ such factors of~$w$ and the sum of their
orders is at least $B^r-B+1$.
By the pigeonhole principle, one of these factors has order at least $B^{r-1}$;
call this factor~$w''$.
Now $w''$ contains at most $r-1$ different letters;
thus $G''=G_{w'',H}(1,2, \ldots, |w''|)$ is an $\ell$-factor of order at least $B^{r-1}$
for which the letter partition has at most $(r-1)$ bags.
By induction, $G''$~contains a strong $\ell$-factor $G_{w', H}(1,2, \ldots, |w'|)$ of order at least~$B$
such that $w'$~is a factor of~$w''$ which is a factor of~$w$.
Hence $w'$~is a factor of~$w$ and we are done.
\end{proof}

\begin{theorem}\label{factors}
Suppose $\X$ is a hereditary class above the Bell number with $k_{\X}$ finite.
Then $\X\supseteq\P(w, H)$ for an infinite almost periodic word~$w$ and a graph~$H$ of order at most~$\ell_{\X}$ with loops allowed.
\end{theorem}

\begin{proof}
From Lemmas~\ref{zero} and \ref{first} it follows that each class~$\X$
with speed above the Bell number with finite distinguishing
number~$k_{\X}$ contains an infinite set~$\S$ of strong $\ell_{\X}$-factors of increasing order.
For each~$H$ on $\{1,2,\dotsc,\ell\}$ with $1\le\ell\le\ell_\X$,
let $\S_H=\{G_{w,H}(1,\dotsc,m)\in\S\}$ be the set of all $\ell_\X$-factors in~$\S$
whose adjacencies are defined using the density graph~$H$.
Then for some (at least one) fixed graph~$H_0$ the set~$\S_{H_0}$ is infinite.
Hence also $L=\{w:G_{w,H_0}(1,\dotsc,m)\in\X\}$ is an infinite language.
As $\X$~is a hereditary class, the language~$L$ is closed under taking word factors (it is a \emph{factorial language}).

It is not hard to see that any infinite factorial language contains an inclusion-minimal infinite factorial language.
So let $L'\subseteq L$ be a minimal infinite factorial language. 
It follows from the minimality that $L'$~is well quasi-ordered by the factor relation,
because otherwise removing one word from any infinite antichain and taking all factors of the remaining words
would generate an infinite factorial language strictly contained in~$L'$.
Thus there exists an infinite chain $w^{(1)},w^{(2)},\dotsc$ of words in~$L'$
such that for any $i<j$, the word $w^{(i)}$~is a factor of~$w^{(j)}$.
More precisely, for each~$i$ there is a non-negative integer~$s_i$ such that $w^{(i)}_k = w^{(i+1)}_{k+s_i}$.
Let $g(i,k)=k+\sum_{j=1}^{i-1}s_j$.
Now we can define an infinite word~$w$ by putting $w_k=w^{(i)}_{g(i,k)}$
for the least value of~$i$ for which the right-hand side is defined.
(Without loss of generality we get that $w$ is indeed an infinite
word; otherwise we would need to take the reversals of all the
words~$w^{(i)}$.)

Observe that any factor of~$w$ is a factor of some~$w^{(i)}$ and hence in the language~$L'$.
If $w$~is not almost periodic, then there exists a factor~$f$ of~$w$
such that there are arbitrarily long factors~$f'$ of~$w$ not containing~$f$.
These factors~$f'$ generate an infinite factorial language~$L''\subset L'$
which does not contain $f\in L'$. This contradicts the minimality of~$L'$ and proves that $w$
is almost periodic.

Because any factor of~$w$ is in~$L$, any $G_{w,H_0}(u_1,\dotsc,u_m)$ is an induced subgraph of some $\ell_\X$-factor in~$\X$.
Therefore $\P(w,H_0)\subseteq\X$.
\end{proof}

Combining Theorems~\ref{thm:ap-min} and~\ref{factors} we derive the main result of this section. 

\begin{cor}
\label{cor:last}
Let $\X$ be a class of graphs with $k_\X<\infty$.
Then $\X$~is a minimal hereditary class above the Bell number if and only if
there exists a finite graph~$H$ with loops allowed and an infinite almost periodic word~$w$
over~$V(H)$ such that $\X=\P(w,H)$.
\qed
\end{cor}

Lastly, note that~-- similarly to the case of infinite distinguishing number~--
each of the minimal classes $\P(w,H)$ has an infinite universal graph:
$G_{w,H}(1,2,3,\dotsc)$.


\section{Decidability of the Bell number}
\label{sec:decid}
As we mentioned in the introduction, every class below the Bell number
can be characterised by a finite set of forbidden induced subgraphs. 
Therefore all classes for which the set of minimal forbidden induced subgraphs is infinite 
have speed above the Bell number. For classes defined by finitely many forbidden 
induced subgraphs, the problem of deciding whether their speed is above 
the Bell number is more complicated and decidability of this problem has been an open question. 
In this section, we employ our characterisation of minimal classes above the Bell number 
to answer this question positively.  

Our main goal is to provide an algorithm that decides for an input
consisting of a finite number of graphs $F_1,\dotsc,F_n$
whether the speed of~$\X=\Free(F_1,\dotsc,F_n)$ is above
the Bell number. That is, we are interested in the following problem.

\begin{problem}
\label{pr:bell}
\probl{A finite set of graphs $\F=\{F_1,F_2,\dotsc,F_n\}$}
{Yes, if the speed of $\X=\Free(\F)$ is above the Bell number; no otherwise.}
\end{problem}

Our algorithm, following the characterisation of
minimal classes above the Bell number, distinguishes two cases
depending on whether the distinguishing number~$k_\X$ is finite or
infinite. First we show how to discriminate between these two cases.

\begin{problem}
\label{pr:dist}
\probl{A finite set of graphs $\F=\{F_1,F_2,\dotsc,F_n\}$}
{Yes, if $k_\X=\infty$ for $\X=\Free(\F)$; no otherwise.}
\end{problem}

\begin{theorem}
\label{thm:dist-poly}
There is a polynomial-time algorithm that solves Problem~\ref{pr:dist}.
\end{theorem}

\begin{proof}
By Theorem~\ref{thm:infinite}, $k_\X=\infty$ if and only if $\X$~contains one of the thirteen minimal classes listed there.
By Theorem~\ref{thm:finforb}, each of the minimal classes is defined by finitely many forbidden induced subgraphs;
thus membership can be tested in polynomial time.
Then the answer to Problem~\ref{pr:dist} is no if and only if each of the minimal classes given by Theorem~\ref{thm:infinite}
contains at least one of the graphs in~$\F$, which can also be tested in polynomial time.
\end{proof}

By Corollary~\ref{cor:last}, the minimal hereditary classes with
finite distinguishing number with speed above the Bell number can
be described as $\P(w,H)$ with an almost periodic infinite word $w$.
That characterisation applies both to classes defined by finitely
many forbidden subgraphs and to classes defined by infinitely many
forbidden subgraphs.
In the case of finitely many forbidden subgraphs, however,
a stronger characterisation is possible, as we show next.

\begin{definition}
\label{def:cyclic}
Let $w=w_1w_2\ldots$ be an infinite word over a finite alphabet~$A$.
If there exists some $p$ such that $w_i=w_{i+p}$ for all $i\in \mathbb{N}$, 
we call the word~$w$ \emph{periodic} and the number~$p$ its \emph{period}.
If, moreover, for some period~$p$ the letters $w_1,w_2,\dotsc,w_p$
are all distinct, we call the word~$w$ \emph{cyclic}.

If $w$ is a finite word, then $w^\infty$ is the periodic word obtained by
concatenating infinitely many copies of the word~$w$;
thus $(w^\infty)_i=w_k$ for $k=i \bmod |w|$.

A class~$\X$ of graphs is called a \emph{periodic class} (\emph{cyclic
class}, respectively) if there exists a graph~$H$ with loops allowed
and a periodic (cyclic, respectively) word~$w$ such that $\X=\P(w,H)$.
\end{definition}

\begin{definition}
\label{def:strip}

Let $A=\{1,2,\dotsc,\ell\}$ be a finite alphabet, $H$~a graph on~$A$ with loops allowed,
and $m$ a positive integer.
Define a graph $S_{H,m}$ with vertex set $V(S_{H,m})=A\times\{1,2,\dotsc,m\}$
and an edge between $(a,j)$ and $(b,k)$ if and only if
$(a,j) \neq (b,k)$ and one of the following holds:
\begin{compactitem}
\item $ab \in E(H)$ and either $|a-b| \neq 1$ or $j \neq k$ (or both);
\item $ab \notin E(H)$ and $|a-b|=1$ and $j=k$.  
\end{compactitem}
The graph $S_{H, m}$ is called an \emph{$(\ell,m)$-strip}.
\end{definition}

Notice that a strip can be viewed as the graph obtained from the union of $m$ disjoint paths
$(1,j){-}(2,j){-}\dotsb{-}(\ell,j)$ for $j \in \{1,2,\dotsc,m\}$ by swapping edges with non-edges
between vertices $(a,j)$ and $(b,k)$ if $ab \in E(H)$. 

\begin{theorem}
\label{thm:charfin}
Let $\X=\Free(F_1,F_2,\dotsc,F_n)$ with $k_\X$ finite.
Then the following conditions are equivalent:
\begin{compactenum}[(a)]
\item The speed of~$\X$ is above the Bell number.
\item $\X$ contains a periodic class.
\item For every $p \in \mathbb{N}$, $\X$ contains a cyclic class with period at least $p$.
\item There exists a cyclic word~$w$ and a graph~$H$ on the alphabet of~$w$
such that $\X$~contains the $\ell$-factor $G_{w,H}(1,2,\dotsc,2\ell m)$
with $\ell=|V(H)|$ and $m=\max\bigl\{|F_i|:i\in\{1,2,\dotsc,n\}\bigr\}$.
\item For any positive integers $\ell$, $m$, the class~$\X$
contains an $(\ell,m)$-strip.
\end{compactenum}
\end{theorem}

\begin{proof}
(a)${}\Rightarrow{}$(b):
From Theorem~\ref{factors} we know that $\X$ contains a class $\P(w,H)$
with some almost periodic word $w=w_1w_2 \ldots$ and a finite graph~$H$ with loops allowed.
Let $m=\max\{|F_1|, |F_2|, \ldots, |F_n|\}$ and let $a=w_1w_2 \ldots w_m$
be the word consisting of the first $m$ letters of the infinite word~$w$.
Since $w$~is almost periodic, the factor~$a$ appears in~$w$ infinitely often.
In particular, there is $m'>m$ such that $w_{m'+1}w_{m'+2} \ldots w_{m'+m}=a$.
Define $b$ to be the word between the two $a$'s in~$w$, i.e., let $b=w_{m+1}w_{m+2}\ldots w_{m'}$.
In this way, $w$~starts with the initial segment~$aba$.

We claim that $\X$ contains the periodic class $\P (w', H)$ with $w'=(ab)^{\infty}$.
For the sake of contradiction, suppose that $\X$ does not contain $\P(w', H)$.
Then for some $i \in \{1,2,\dotsc,n\}$, we have $F_i \in \P(w', H)$.
So $F_i \cong G_{w',H}(u_1, u_2, \ldots, u_k)$ for some $u_1<u_2< \dotsb < u_k$. 
Let $U=\{u_1,u_2,\dotsc,u_k\}$.
We are now looking for a monotonically increasing function $f:U\to\mathbb{N}$ with these two properties:
firstly, $w_{f(u)}=w'_{u}$ for any $u\in U$; secondly, $f(u)-f(u')=1$ if and only if $u-u'=1$.
If we can establish the existence of such a function, we will then have
$F_i \cong G_{w',H}(u_1,u_2,\dotsc,u_k) \cong G_{w,H}(f(u_1),f(u_2),\dotsc,f(u_k)) \in
\P(w,H) \subseteq \X$, a contradiction.

To construct such a function~$f$, consider a maximal block $\{u_j,u_{j+1},\dotsc,u_{j+p}\}$ of consecutive integers in~$U$
(that is, $u_{j-1}<u_j-1$; $u_j=u_{j+1}-1$;~\dots; $u_{j+p-1}=u_{j+p}-1$; $u_{j+p}<u_{j+p+1}-1$).
Furthermore, consider the word $w'_{u_j} w'_{u_{j+1}} \ldots w'_{u_{j+p}}$ of length at most~$m$,
which is a factor of~$w'=(ab)^\infty$ and thus also a factor of~$aba$ because $|aba|>2m$.
The word $aba$, being a factor of~$w$, appears infinitely often in~$w$ because $w$~is almost periodic.
Hence not only we can define $f: U\to\mathbb{N}$ in such a way that $w_{f(u)}=w'_{u}$ for any $u\in U$
and that blocks of consecutive integers in~$U$ are mapped to blocks of consecutive integers,
but we can also do it monotonically and so that $f(u)>f(u')+1$ whenever $u>u'+1$.
This finishes the proof of the first implication.

\smallskip
(b)${}\Rightarrow{}$(c):
Let $\X$ contain a class $\P(w,H)$, where $w$~is a periodic word and $H$~is a graph with loops allowed on the alphabet of~$w$. 
If $w$ is not cyclic (i.e., if some letters appear more than once within the period), 
the class $\P(w,H)$ can be transformed into a cyclic one by extending the alphabet, renaming multiple appearances
of the same letter within the period to different letters of the extended alphabet and modifying the graph $H$ accordingly. 
More formally, let $p$ be the period of $w$. Define a new word $w'=(123\dots p)^\infty$ and a graph~$H'$
with vertex set $\{1,2,\dotsc,p\}$ and an edge $ij\in E(H')$ if and only if $w_iw_j\in E(H)$.
Then any graph $G_{w,H}(u_1,u_2,\dotsc,u_m)\in\P(w,H)$ is isomorphic to the graph
$G_{w',H'}(u_1,u_2,\dotsc,u_m)\in\P(w',H')$. Hence $\X$~also
contains~$\P(w',H')$, where the word~$w'$ is cyclic. 

Since any periodic word of period $p$ is also a periodic word of period $kp$ for any $k \in \mathbb{N}$, 
by the same transformation a periodic class of period $p$ can be transformed into a cyclic class of period $kp$.

\smallskip
(c)${}\Rightarrow{}$(d): 
Follows directly from the definition of a cyclic class.

\smallskip
(d)${}\Rightarrow{}$(a): 
Let $\X$~contain the $\ell$-factor $G_{w,H}(1,2,\dotsc,2\ell m)$.
We prove that $\X$ then contains the whole class~$\P(w,H)$; its
speed will then be above the Bell number by Proposition~\ref{prop:pwh}.
For the sake of contradiction suppose that for some~$i$, the graph~$F_i$ belongs to~$\P(w,H)$.
Then $F_i \cong G_{w,H}(u_1, u_2, \ldots, u_k)$ for some $k \leq m$ and $u_1<u_2< \dotsb < u_k$.
Note that the period of~$w$ is~$\ell$ (by the definition of a cyclic class). Put $u'_1=u_1\bmod\ell$,
or $u'_1=\ell$ if $u_1\equiv 0 \pmod\ell$.
Furthermore, for each $i\ge2$ let $u'_i=(u_i\bmod\ell)+c_i\ell$, where each $c_i$~is chosen in such a way that
$0<u'_i-u'_{i-1}\le\ell+1$ for all~$i$ and $u'_i-u'_{i-1}=1$ if and only if $u_i-u_{i-1}=1$.
By construction, $F_i\cong G_{w,H}(u_1,u_2,\dotsc,u_k)\cong G_{w,H}(u'_1,u'_2,\dotsc,u'_k)$ and
$u'_k<k(\ell+1)\le m(\ell+1)\le 2\ell m$.
Hence $F_i$~is isomorphic to an induced subgraph of $G_{w,H}(1,2,\dotsc,2\ell m)\in\X=\Free(F_1,F_2,\dotsc,F_n)$,
a contradiction.

\smallskip
(c)${}\Rightarrow{}$(e): 
Let $\X$ contain the cyclic class $\P(w,H)$, where $w$~is a cyclic word with period~$p>\ell$.
Since $p>\ell$, the subgraph of $G_{w,H}(1,2,\dotsc,pm)$ induced by the bags
corresponding to the first $\ell$ letters of~$w$ is an $(\ell,m)$-strip.

\smallskip
(e)${}\Rightarrow{}$(a): 
Let $\ell_\X$, $d_\X$ and $c_\X$ be the constants given by Lemma~\ref{finitedistinguishingnumber}
and put $m=2d_\X+c_\X+2$.
We show that for any fixed positive integer~$\ell$, the class~$\X$ contains
a strong $(\ell_\X,d_\X)$-graph~$G$ such that its sparsification~$\phi(G)$ has a connected component of order at least~$\ell$.
Then we can apply Theorem~\ref{paths}.

So let $\ell$ be a positive integer. By assumption, $\X$~contains an $(\ell,m)$-strip~$S_{H,m}$.
By Lemma~\ref{finitedistinguishingnumber}, after removing no more than $c_\X$~vertices
we are left with a strong $(\ell_{\X},d_{\X})$-graph~$S'$ with a strong $(\ell_{\X}, d_{\X})$-partition~$\pi$.
Let $V_a=\{(a,j)\in V(S'): 1\le j\le m\}$ be the letter bags of~$S'$, $1\le a\le\ell$,
and consider the prime partition $p(\pi)=\{W_1, W_2, \dotsc, W_{\ell'}\}$.
If two vertices $x, y$ belong to different bags of~$p(\pi)$, then according to Lemma~\ref{symdif2}
we have $|N(x) \ominus N(y)| \geq 5 \times 2^{\ell_\X}d_\X-2d_\X \geq 8$.
However, if we have two vertices $(a,j)$, $(a,j')$ of~$S'$ in the same letter bag~$V_a$,
then  $N(a,j)\ominus N(a,j') \subseteq \{({a-1},j), ({a-1},j'),\allowbreak
(a,j), (a,j'),\allowbreak ({a+1},j), ({a+1},j')\}$, so its size is at most~$6$.
Hence, we deduce that each $V_a \subseteq W_{f(a)}$ for some function~$f$.

Now notice that $(V_a, V_b)$ is $d_\X$-dense, that is, $ab$~is and edge of~$H$,
if and only if $(W_{f(a)}, W_{f(b)})$ is $d_{\X}$-dense.
Indeed, if one of them was $d_{\X}$-dense and the other $d_{\X}$-sparse,
then $(V_a, V_b)$ would be both $d_{\X}$-dense and $d_{\X}$-sparse, in which case $|V_a| \leq 2d_{\X}+1$.
But this is not true, as $V_a$~is obtained from a set of size $m= 2d_{\X} + 2 + c_{\X}$
by removing at most $c_{\X}$~vertices.

It follows that $\phi(S')$ is constructed by swapping edges with non-edges between $V_a$ and~$V_b$ such that $ab\in E(H)$.
Hence $\phi(S')$ is a linear forest obtained from the paths $(1,j){-}(2,j){-}\dotsb{-}(\ell,j)$ for $j \in \{1,2,\dotsc,m\}$
by removing at most $c_\X$ vertices.
As $m>c_{\X}$, at least one of the paths is left untouched. Therefore, $\phi(S')$~contains a connected component of size at least~$\ell$.
\end{proof}

Finally, we are ready to tackle the decidability of Problem~\ref{pr:bell}.

\begin{algo}
\label{algo:bell}
\probl{A finite set of graphs $\F=\{F_1,F_2,\dotsc,F_n\}$}
{Yes, if the speed of $\X=\Free(\F)$ is above the Bell number; no otherwise.}
\begin{enumerate}[(1)]
\item Using Theorem~\ref{thm:dist-poly}, decide whether $k_\X=\infty$. If it is, output \emph{yes} and stop.

\item Set $m:=\max\{|F_1|,|F_2|,\dotsc,|F_n|\}$ and $\ell:=1$.

\item Loop:
\begin{enumerate}[(3a)]
\item For each graph (with loops allowed) $H$ on $\{1,2,\dotsc,\ell\}$ construct the $(\ell,\ell)$-strip $S_{H,\ell}$.
	Check if some $F_i$ is an induced subgraph of~$S_{H,\ell}$. If for each~$H$ the strip~$S_{H,\ell}$ contains
	some~$F_i$, output \emph{no} and stop.
\item For each graph (with loops allowed) $H$ on $\{1,2,\dotsc,\ell\}$ and for each word~$w$ consisting of $\ell$
	distinct letters from $\{1,2,\dotsc,\ell\}$ check if the $\ell$-factor $G_{w^\infty,H}(1,2,\dotsc,2\ell m)$
	contains some~$F_i$ as an induced subgraph. If one of these $\ell$-factors contains no~$F_i$, output \emph{yes} and stop.
\item Set $\ell := \ell+1$ and repeat.
\end{enumerate}
\end{enumerate}
\end{algo}

It remains to prove the correctness of this algorithm.

\begin{theorem}
Algorithm~\ref{algo:bell} correctly solves Problem~\ref{pr:bell}.
\end{theorem}

\begin{proof}
We show that if the algorithm stops, it gives the correct answer,
and furthermore that it will stop on any input without entering an infinite loop.
First, if it stops in step (1), the answer is correct by~\cite{jump}, since any class with
infinite distinguishing number has speed above the Bell number.

Assume that the algorithm stops in step (3a) and outputs \emph{no}.
This is because every $(\ell,\ell)$-strip contains some forbidden subgraph~$F_i$,
hence no $(\ell,\ell)$-strip belongs to~$\X$.
By Theorem~\ref{thm:charfin}(e), the speed of~$\X$ is below the Bell number.

Next suppose that the algorithm stops in step (3b) and answers \emph{yes}.
Then $\X$~contains the $\ell$-factor $G_{w^\infty,H}(1,2,\dotsc,2\ell m)$, where 
$w^\infty$~is a cyclic word. Hence by Theorem~\ref{thm:charfin}(d)
the speed of~$\X$ is above the Bell number.

Finally, if $k_\X=\infty$ the algorithm stops in step~(1).
If $k_\X<\infty$ and the speed of~$\X$ is above the Bell number,
then by Theorem~\ref{thm:charfin}(d) the algorithm will stop in step (3b).
If, on the other hand, the speed of~$\X$ is below the Bell number,
then by Theorem~\ref{thm:charfin}(e) there exist positive integers
$\ell$, $M$ such that $\X$~contains no $(\ell,M)$-strip.
Let $N=\max\{\ell,M\}$. Obviously, $\X$~contains no $(N,N)$-strip,
because any $(N,N)$-strip contains some (many) $(\ell,M)$-strips as induced subgraphs
and $\X$~is hereditary.
Therefore the algorithm will stop in step~(3a) after finitely many steps.
\end{proof}


\section{Concluding remarks}
\label{sec:con}

In this paper, we have characterised all minimal hereditary classes
of graphs whose speed is at least the Bell number~$B_n$.
This characterisation allowed us to show that the problem of
determining if the speed of a hereditary class~$\X$ defined by
finitely many forbidden induced subgraphs is above or below the Bell
number is decidable, i.e., there is an algorithm that gives a
solution to this problem in a finite number of steps.
However, the complexity of this algorithm, in terms of the input
forbidden graphs, remains an open question.
In particular, it would be interesting to determine if there is a
polynomial bound on the minimum~$\ell$ such that the input class~$\X$
contains an $\ell$-factor as in Theorem~\ref{thm:charfin}(d) if it
is above the Bell number, and it fails to contain any $(\ell,\ell)$-strip
as in Theorem~\ref{thm:charfin}(e) if it is below.

\end{document}